\documentclass[12pt]{article}

\usepackage{amsmath,amsthm,amssymb,graphicx}

\def\N{\mathbb N}
\def\N0{\mathbb N_0}
\def\Z{\mathbb Z}

\def\real{\mathbb R}

\def\alg{\mathbb A}
\def\Balg{\mathbb B}

\def\graph{{\mathcal G}}

\def\vertexset{{\mathcal V}}

\def\edgeset{{\mathcal E}}

\def\graphbar{\overline{\graph}}

\def\wt{\mu}

\def\cond{C}
\def\rest{R}

\def\dop{{\mathcal L}}
\def\sg{{\mathcal S}}
\def\domain {{\mathcal D}}

\title{\bf Myopic Models of Population Dynamics on Infinite Networks} 
\author{Robert Carlson \\
Department of Mathematics \\ 
University of Colorado at Colorado Springs \\
carlson@math.uccs.edu}

\newtheorem{thm}{Theorem}[section]
\newtheorem{cor}[thm]{Corollary}
\newtheorem{lem}[thm]{Lemma}
\newtheorem{prop}[thm]{Proposition}

\theoremstyle{definition}

\theoremstyle{remark}

\newcommand{\thmref}[1]{Theorem~\ref{#1}}

\newcommand{\lemref}[1]{Lemma~\ref{#1}}
\newcommand{\corref}[1]{Corollary~\ref{#1}}
\newcommand{\propref}[1]{Proposition~\ref{#1}}

 \numberwithin{equation}{section}

\begin{document}

\maketitle

\begin{abstract}

Reaction-diffusion equations are treated on infinite networks using semigroup methods.
To blend high fidelity local analysis with coarse remote modeling,
initial data and solutions come from a uniformly closed algebra generated by functions 
which are flat at infinity.  The algebra is associated with a compactification of the network
which facilitates the description of spatial asymptotics.  Diffusive effects disappear at infinity,
greatly simplifying the remote dynamics.  Accelerated diffusion models with conventional 
eigenfunction expansions are constructed to provide opportunities for finite dimensional approximation.

\end{abstract}

\vskip 25pt



\newpage

\tableofcontents

\newpage 

\section{Introduction}

Recent research has drawn attention to a wide variety of massive networks with evolving populations or states \cite{NBW}.
These may be social networks, in which diseases \cite{Keeling,Newman}, opinions, or knowledge propagate, 
computer networks with spreading malware, or spatial environments such as road networks
or river systems \cite{Ramirez,SCA}.  Classical studies of the dynamics of interacting populations (or chemical species) in a spatial domain
often utilize nonlinear reaction diffusion partial differential equations  \cite{Fife, Murray1, Murray2}, having the form
\begin{equation} \label{demodel}
\frac{dp}{dt} + \Delta p = f(p),
\end{equation}
with the Laplacian $\Delta $ generating the diffusion, while $f(p)$ or $f(t,p)$ describes the growth and interaction of several
populations at a site. This work treats reaction diffusion models on infinite discrete networks,
with the classical population density $p(t,x)$ replaced by a function $p(t,v)$ on the network vertices.
The role of the diffusion operator $\Delta $ will be played by a spatial difference operator.  
There has been some previous work considering reaction-diffusion equations on networks \cite{Colizza},
apparently with aims different from those considered here. 

Massive networks can be so enormous or structurally complex that finiteness may be a largely irrelevant property
to a network modeler with limited information about the remote network structure or vertex states.
In this work,  infinite networks are combined with restricted population models
to describe population dynamics with simple behavior 'at infinity'.
The viewpoint emphasizes high fidelity local modeling with coarse descriptions of distant features,
somewhat like the famous Saul Steinberg cartoon (New Yorker magazine, March 29,1976)  which 
shows a myopic map of the United States as seen from Ninth Avenue in Manhattan.
The fidelity of features decreases rather dramatically with the distance from Manhattan,
with buildings on Tenth Avenue presented in detail, Texas acknowledged only by name, 
and numerous states completely absent.  

A thought experiment may help motivate the coming developments.
Suppose a public health official in, say, Iceland wants to model the local impact of a global epidemic 
dispersed by physical contact.  Each person is represented by a graph vertex, while people who have physical contact
are joined by a graph edges.  Edges have weights (distances) which are inversely proportional to  
the frequency of contact or transmission rate.  The model tracks the population of pathogens
for each person, using a system \eqref{demodel}.
A lack of information or a desire to manage the complexity of the model
leads the official to divide the rest of the world into weakly interacting blocks, in which
the pathogen load is assumed the same for each individual.  Thus North America,
India, and Japan may each be considered a block with distinct estimates of average pathogen load.     

Two aspects of such modeling will play an important role in the subsequent developments.
First, the granularity of the block structure should be adjustable: perhaps major urban areas 
or physically separated regions within countries could be modeled as distinct blocks.
Second, since the blocks are large but weakly interacting, it is expected that remote coarsely modeled parts of the social network
will feel the diffusive influence of other blocks slowly, with the local dynamics described by $f(p)$ playing a dominant role.
These two aspects are treated by selecting a space of 'eventually flat' functions as the starting point for
population modeling, and by using an infinite network, which facilitates the description and analysis of 'remote'
vertices.  Banach space semigroup methods will be used to analyze these models.  When the initial class of population states
is completed in the $l^{\infty}$ norm, the resulting Banach algebra can be identified with
the continuous functions on a compactification of the network, which simplifies the analysis of the 
population dynamics modeled by \eqref{demodel}.  

The analysis of \eqref{demodel} on networks begins in next section with the introduction of 
locally finite infinite graphs with edge weights, vertex weights, and their classical Banach spaces.
The subalgebra $\alg \subset l^{\infty} $ is introduced.  This is the algebra of real valued functions on the vertex set
whose values only change across a finite set of edges.  A second algebra $\Balg $ is obtained by completing $\alg $
in the $l^{\infty }$ norm.   Diffusions are generated by local 'Laplace' operators $\Delta $ which depend on the edge and vertex weights. 
The emphasis is on bounded operators $\Delta $.
Basic properties of these operators and the semigroups $\exp (-t\Delta )$ they generate are reviewed.
The results here are known, but are scattered in the literature.  Proofs are outlined so the material is readily
accessible.  The section ends with the observation that $\Balg $ is an invariant subspace for $\Delta $.

The algebras $\alg $ and $\Balg $ are studied in the third section.  $\Balg $ is naturally identified with the space of all continuous
functions on a compactification $\graphbar $ of $\graph $ which is obtained by completing $\graph $ in a new metric satisfying the
condition that the sum of the edge lengths is finite.  The compactification $\graphbar $ is a totally disconnected compact metric space.
The closed ideals of $\Balg $ obtained by requiring functions to vanish on closed subsets of the boundary of $\graphbar $
are also invariant subspaces for $\Delta $.

The fourth section begins with a treatment of \eqref{demodel} on the compactification $\graphbar $.
On the boundary of $\graphbar $, the solutions are not affected by the diffusion; the dynamics simply follows the
reduced system  of ordinary differential equations.    Operators $\Delta _{\rho }$ associated to the modified metrics 
mentioned above are treated.  These are typically unbounded, with domains determined by vanishing conditions on 
closed subsets of the boundary of $\graphbar $.   When the vertex weights also satisfy a finiteness condition, the operators
$\Delta _{\rho}$ have compact resolvent, and eigenfunction expansions are available. 
If the metrics $\rho $ agree with the original metric on suitable subgraphs,
the semigroups $\exp(-t \Delta _{\rho })$ converge strongly to $\exp (-t \Delta )$.  
This convergence extends to solutions of the corresponding nonlinear equations \eqref{demodel}.

Funding supporting this work was provided by a seed grant from the University of Colorado at Colorado Springs
Center for the BioFrontiers Institute.
The author thanks R. Cascaval and Y. Zhang for helpful conversations.

\section{Networks and some local operators}

This section introduces most of the main objects of our study:  locally finite weighted graphs, 
Banach spaces of functions on the vertex set, 'Laplace' operators $\Delta $ acting on functions, and the semigroups $\exp (-t\Delta )$.
The operators $\Delta $ and  $\exp (-t\Delta )$ are treated on the classical $l^p$ spaces,
with particular focus on two subalgebras of $l^{\infty}$, $\alg $ and $\Balg $.
Bounded operators $\Delta $ are emphasized for most of this work.
In addition to introducing basic concepts, this section includes a number of previously known 
results whose proofs are scattered in the literature.

\subsection{Networks}

In this work a network or graph $\graph $ (the terms will be used interchangeably)  
will have  a countable vertex set $\vertexset$ and edge set $\edgeset$.  
Edges are undirected, with $u \sim v$ meaning the unordered pair $[u,v]$ is in $ \edgeset$.
$\graph $ is simple; there are no loops ($u \not= v$), and at most one edge joins a 
vertex pair.  Each vertex will have at least one and at most finitely many incident edges.  
General references on graphs are \cite{Chung,Diestel}.

$\graph $ is equipped with a vertex weight function $\wt : \vertexset \to (0,\infty )$, 
and an edge weight function $\rest :\edgeset \to (0,\infty )$ whose values are denoted by
$\rest (u,v)$ when $[u,v] \in \edgeset $.  Since the edges are undirected, $\rest (u,v) = \rest (v,u)$. 
Edge weights are commonly identified with electrical network resistance \cite{Doyle} or \cite{Lyons}.  
Edge conductance is $\cond (u,v) = 1/\rest (u,v) $ if $\rest(u,v) > 0$, and $0$ otherwise.  

A finite path $\gamma $ in $\graph $ connecting vertices $u$ and $v$ 
is a finite sequence of vertices
$u=v_0, v_1,\dots ,v_K = v$ such that $[v_k,v_{k+1}] \in \edgeset $ for $k = 0,\dots ,K-1$. 
$\graph $ is connected if there is a finite path from $u$ to $v$ for all $u,v \in \vertexset$.  
The edge weights provide a (geodesic) metric on $\graph $, defined by   
\begin{equation} \label{metricdef}
d_{\rest} (u,v) = \inf_{\gamma } \sum_k \rest (v_k, v_{k+1}), 
\end{equation}
the infimum taken over all finite paths  $\gamma $ joining $u$ and $v$.
This metric space has an extension to a complete metric space \cite[p. 147]{Royden},
which will be denoted $\graphbar _{\rest}  $, or simply $\graphbar $ if the choice of $\rest $ is clear.

Extending the combinatorial notion of path, a path in $\graphbar $
will be a sequence $\{ v_k \}$ with $v_k \in \vertexset $, $[v_k,v_{k+1}] \in \edgeset$,
where the index set may be finite (finite path), the positive integers 
(a ray), or the integers (a double ray).  The role of continuous paths in $\graphbar $  
is played by paths going from $u \in \graphbar $ to $v \in \graphbar $, 
which in the double ray case requires
$\lim _{k \to -\infty} d_{\rest}(v_k,u) = 0$ and $\lim _{k \to \infty} d_{\rest}(v_k, v) = 0$. 
The ray case is similar.
A path for which all vertices are distinct is a simple path.     
If $\graph $ is connected then there is a path joining any pair of points $u,v \in \graphbar $.

\subsection{Operators $\Delta $ on $\graph $}

With respect to a vertex weight $\wt $ the $l^p$ norms of functions $f:\vertexset \to \real $ are 
\[\| f \| _p = \Bigl ( \sum_{v \in \vertexset} |f(v)|^p \wt(v) \Bigr ) ^{1/p} ,\quad 1 \le p < \infty ,\]
and
\[\| f \| _{\infty } = \sup_{v \in \vertexset} |f(v)|  .\]
Functions $f$ with finite norm comprise the $l^p$ spaces. 
In particular the Hilbert space $l^2$ consists of the functions 
$f:\vertexset \to \real $ with $\sum_{v \in \vertexset} |f(v)|^2 \wt (v) < \infty $, with 
the inner product $\langle f,g \rangle  = \sum_{v \in \vertexset} f(v) g(v) \wt (v)$.  

The Banach space $l^{\infty}$ is an algebra with pointwise addition and multiplication.
Two subalgebras of $l^{\infty}$ are well matched to the goal of myopic modeling when $\graph $ is infinite and connected.
The first algebra $\alg \subset l^{\infty} $ is defined as the algebra of functions $f :\vertexset \to \real $
such that the set of edges $[u,v]$ in $\edgeset $ with $f(u) \not= f(v)$ is finite.   
If $\graph $ is infinite, $\alg$ will not be closed in $l^{\infty}$.  Let $\Balg $ denote the closure of $\alg$ in $l^{\infty }$.

Formal operators $\Delta $, analogous to the Laplace operator of a Riemannian manifold, are defined by  
\begin{equation} \label{Lapop}
\Delta f(v) = \frac{1}{\wt (v)}\sum_{u \sim v} C(u,v) (f(v) - f(u)).
\end{equation}        
The vector space $\domain _K$ of functions $f:\vertexset \to \real $ which are $0$ 
at all but finitely many vertices provides an initial domain for $\Delta $.
If the vertex weights satisfy $\sum_{v \in \vertexset} \wt (v) < \infty $, then all
bounded functions will lie in $l^p$ for $1 \le p < \infty $. 
If $\sum_{v \in \vertexset} \wt (v) = \infty $ there may  be square integrable elements of $\alg $;
define $\domain _{\alg } = \alg  \cap l^2$.
The next proposition collects basic facts about the symmetric bilinear forms 
induced by the edge conductances $C(u,v)$.
Closely related results using the smaller domain $\domain _K$
are in \cite[p. 20]{DaviesHKST} and \cite{KL2}.

\begin{prop} \label{bform}
 For $f \in \domain _{\alg } $, define the operator
\begin{equation} \label{Laplace}
\Delta f(v) = \frac{1}{\wt (v)}\sum_{u \sim v} C(u,v) (f(v) - f(u)).
\end{equation}
The symmetric bilinear form
\begin{equation} \label{bformdef}
B(f,g) = \frac{1}{2}\sum_{v \in \vertexset} \sum_{u \sim v} C(u,v)(f(v) - f(u))(g(v) - g(u)),
\quad f,g \in \domain _{\alg}, 
\end{equation}
has a nonnegative quadratic form $B(f,f)$, and satisfies
\begin{equation} \label{symmetry}
B(f,g) = \langle \Delta f,g \rangle  = \langle f, \Delta g \rangle , \quad f,g \in \domain _{\alg}.
\end{equation}

\end{prop}

\begin{proof}

The nonnegativity of the quadratic form is immediate from the definition.
Note that for any $f \in \alg $ there are only finitely many vertices $v \in \vertexset $ 
for which $f(v) - f(u)$ is ever nonzero if $u$ is adjacent to $v$.

To relate the operator $\Delta $ and the form $B$, start with
\begin{equation} \label{compform}
2B(f,g) 
 = \sum_{v \in \vertexset} g(v) \sum_{u \sim v} C(u,v) \Bigl ( f(v) - f(u) \Bigr )
\end{equation}
\[ - \sum_{v \in \vertexset} \Bigl ( \sum_{u \sim v} C(u,v) g(u)(f(v) - f(u)) \Bigr )\]
Suppose a graph edge $e$ has vertices $v_1(e)$ and $v_2(e)$.  
The second sum over $v \in \vertexset$ in \eqref{compform} can be viewed as a 
sum over edges, with each edge contributing the terms
$C(v_1,v_2)g(v_1)(f(v_2) - f(v_1))$ and $C(v_1,v_2)g(v_2)(f(v_1) - f(v_2))$.
Using this observation to change the order of summation gives
\[ \sum_{v \in \vertexset} \Bigl ( \sum_{u \sim v} C(u,v) g(u)(f(v) - f(u)) \Bigr )\]
\[ = \sum_{e \in \edgeset} C(v_1(e),v_2(e))
\Bigl (g(v_1)(f(v_2) - f(v_1)) + g(v_2)(f(v_1) - f(v_2)) \Bigr )\]
\[ = \sum_{u \in \vertexset} g(u) \sum_{v \sim u} C(u,v) (f(v) - f(u))  \]
Employing this identity in \eqref{compform} gives
\begin{equation} \label{formform}
2B(f,g) 
= 2 \sum_v \wt (v) g(v) [\frac{1}{\wt (v)}\sum_{u \sim v} C(u,v) (f(v) - f(u))].
\end{equation}

\end{proof}

\subsubsection{Relation to Markov chains}

A continuous time Markov chain uses a system of constant coefficient differential equations 
\begin{equation} \label{kolbeq}
\frac{dP}{dt} = QP, \quad P(0) = I.
\end{equation}
to describe the evolution of probability densities $X(t) = X(0)P(t)$ 
on a countable set of states. 
An associated graph (generally directed) may be constructed by connecting states (vertices) $i$ and $j$
with an edge if $Q_{ij} \not= 0$.

With respect to the standard basis consisting of functions $1_w: \vertexset \to \real $
with $1_w(w) = 1$ and $1_w(v) = 0$ for $v \not= w$, 
the operators $\Delta $ have the matrix representation 
\[Q(v,w) = \Bigl \{ \begin{matrix} \wt ^{-1}(w) \sum_{u \sim w} C(u,w), & v = w \cr 
-\wt ^{-1}(v) C(v,w), & v \sim w \cr 
0, & {\rm otherwise } \end{matrix} \Bigr \},
\quad v,w \in \vertexset . \]
If $v$ is fixed, then summing on $w$ gives
\[\sum_w Q(v,w) = \frac{1}{\wt (v)} \sum_{u \sim v} C(u,v) 
- \frac{1}{\wt (v)} \sum_{u \sim v} C(u,v) = 0,\]
so $-Q(v,w)$ is a $Q$-matrix in the sense of Markov chains \cite[p. 58]{Liggett}.
In the $Q$ - matrix formulation the matrix entries represent transition rates.

In the finite state case the solution of \eqref{kolbeq} is simply $P(t) = e^{Qt}$.
When the set of states is infinite the formal description of the operator $Q$ 
may be inadequate to determine the desired semigroup $e^{Qt}$, an issue known 
in probability as the problem of explosions.

\subsubsection{Bounded operators}

Using the usual Banach space operator norms, 
\[ \| \Delta  \|_{p} = \sup_{\| f \| _p \le 1} \| \Delta f \| _p, \quad 1 \le p \le \infty ,\] 
the next proposition characterizes the operators $\Delta $ which are bounded
on all the $l^p$ spaces, $1 \le p \le \infty $.  This situation offers considerable technical advantages for developing
the properties of the semigroup $\exp (-t\Delta )$.   A similar result is in \cite{Davies93}.

\begin{prop}
The norms for the operator $\Delta $ satisfy
\[ \| \Delta  \|_{\infty} = \sup_{v \in \vertexset } \frac{2}{\wt (v)} \sum_{u \sim v} C(u,v) \]
and 
 \[ \| \Delta  \| _1 \le   \sup_v \frac{2}{\wt (v)}\sum_{u \sim v} C(u,v).\]
Consequently, the condition 
\begin{equation} \label{lpbnd}
 \sup_{v \in \vertexset } \frac{1}{\wt (v)} \sum_{u \sim v} C(u,v) < \infty 
 \end{equation}
is equivalent to $\Delta $ being bounded on all the $l^p$ spaces, $1 \le p \le \infty $. 
\end{prop}

\begin{proof}
Consider $ \| \Delta  \|_{\infty}$ first.
 If $|f(u)| \le 1$ for all $u \in \vertexset $, then
\[ |\Delta f(v)| \le \frac{2}{\wt (v)}\sum_{u \sim v} C(u,v) ,\]
while if $f(v) = 1$ and $f(u) = -1$ for all $u \sim v$ then
\[ |\Delta f(v)| = \frac{2}{\wt (v)} \sum_{u \sim v} C(u,v) .\]

Turning to $ \| \Delta  \|_{1}$, suppose $\| f \| _1 < \infty $.  Then  
\[  \| \Delta  f \| _1= \sum_v  \wt (v) |\frac{1}{\wt (v)}\sum_{u \sim v} C(u,v) (f(v) - f(u))|  \]
\[ \le \sum_{[u,v] \in \edgeset}  2C(u,v)(|f(v)| + |f(u)|) 
   \le 2\sum_v |f(v)| \bigl (\sum_{u \sim v} C(u,v) \bigr ) \]
\[  \le 2\sum_v \wt (v)|f(v)| \frac{1}{\wt (v)} \bigl (\sum_{u \sim v} C(u,v) \bigr ) 
 \le \sup_v \bigl (\frac{2}{\wt (v)}\sum_{u \sim v} C(u,v) \bigr ) \| f \| _1 .\]

Bounds for the remaining $l^p$ spaces follow from an elementary case of the Reisz-Thorin Interpolation Theorem \cite[p. 3]{DaviesHKST}. 

\end{proof}

\subsection{The semigroup $\sg (t) = \exp(-t \Delta ) $} 

\subsubsection{Basic properties}

If  \eqref{lpbnd} is satisfied, then the $l^p$ semigroup of bounded operator-valued functions $\exp(-t \Delta ) $ 
can be defined \cite[pp. 1-3]{Pazy} by a power series convergent in the operator norm,
\[\sg (t) = \exp(-t\Delta ) = \sum_{n=0}^{\infty} (-t\Delta )^n/n! , \quad t \ge 0.\]
Various well known 'heat equation' properties of the $l^p$ semigroups $\exp(-t \Delta ) $ are collected in the next proposition.
Additional information can be found in \cite[p. 13-16]{DaviesHKST}.

\begin{prop} \label{heatprop}
Assume that \eqref{lpbnd} is satisfied.   For $1 \le p \le \infty $,
the $l^p$ semigroup  $\sg (t) = \exp(-t \Delta ) $ is positivity preserving.
If $\graph $ is connected, $f \ge 0$, and $f(v) > 0$ for some $v \in \vertexset $,
then $S(t)f(w) > 0$ for all $t > 0$ and $w \in \vertexset $. 
The semigroup $\exp (-t\Delta )$ is a contraction on $l^p$ for $1 \le p \le \infty $,
and preserves the $l^1$ norm of nonnegative functions.
\end{prop}

\begin{proof}
Using an argument from \cite[pp. 68-71]{Liggett}, split $\Delta $ as the difference of the diagonal and off-diagonal parts $\Delta  = D  - N$,
where these two bounded operators are
\[Df(v) =  \frac{1}{\wt (v)} [\sum_{u\sim v} C(u,v)] f(v), \quad  Nf(v) =  \frac{1}{\wt (v)} \sum_{u\sim v} C(u,v) f(u).\]
The differential equation 
\[\frac{d}{dt} \sg (t) + \Delta \sg(t) = 0 \]
for this semigroup may be written as  
\[\frac{d}{dt} \sg (t) + D\sg (t) = N\sg (t), \quad {\rm or} \quad 
\frac{d}{dt} \exp(tD)\sg (t)  = \exp(tD) N\sg (t),\]
and integration gives the following representation for $\sg (t) f$,
\begin{equation} \label{inteq}
\sg (t)f  = \exp(-tD)f + \int_0^t  \exp((s-t)D) N\sg (s) f \ ds .
\end{equation}
The operators $N$ and $\exp (-tD)$ both preserve nonnegative functions, so solving \eqref{inteq} by iteration
shows that $\sg (t)$ is positivity preserving. 

Introduce the normalized functions $\delta _v: \vertexset \to \real $ with 
$\delta _v(v) = 1/\wt (v)$ and $\delta _v(w) = 0$ for $v \not= w$.  For $u \sim v$, the formula for $N$ gives
\[ N\delta _v (u) = \frac{1}{\wt (u)} C(v,u) \delta _v(v)  > 0 .\]
It then follows from \eqref{inteq} that if $\graph $ is connected $f \ge 0$, and $f(v) > 0$ for some $v \in \vertexset $,
then 
\begin{equation} \label{strictpos}
S(t)f(w) > 0, \quad t > 0, \quad  {\rm for \ all \ } w \in \vertexset . 
\end{equation}

The operator $\Delta $ annihilates the constants, so $\exp (-t\Delta ) 1 = 1$.
If $f \in l^{\infty }$ is nonnegative then for all $v \in \vertexset $
\[0 \le \exp(-t\Delta ) f (v) = \exp(-t\Delta ) \| f \|_{\infty }(v) + \exp(-t\Delta ) (f - \| f \|_{\infty })(v) \]
\[ =  \| f \|_{\infty } + \exp(-t\Delta ) (f - \| f \|_{\infty })(v) .\]
The last term is nonpositive, so $\exp (-t\Delta )$ reduces the $l^{\infty}$ norm of nonnegative functions.
If $g \in l^{\infty}$ is written as the difference of two nonnegative functions, 
\[g = g^+ - g^-, \quad g^+(v) = \Bigl \{ \begin{matrix} g(v), & g(v) > 0, \cr
0,& g(v) \le 0, \end{matrix} \Bigr \} \]
then
\[\exp(-t\Delta ) g = \exp(-t\Delta ) g^+ - \exp(-t\Delta ) g^-, \]
so $\exp (-t\Delta )$ is a contraction on $l^{\infty} $.
 
 Using \propref{bform}, the simple  calculation  
\begin{equation} \label{contract2}
\frac{d}{dt} \langle \sg (t)f,\sg (t)f \rangle  = - 2\langle \Delta \sg (t)f, \sg (t)f \rangle 
 = -2 B(\sg (t) f,\sg (t) f ) \le 0
 \end{equation} 
shows that $\sg (t) = \exp (-t \Delta )$ acts by contractions on $l^2$.

Suppose $\Omega \subset \vertexset $ is finite, and $1_{\Omega }$ denotes the characteristic function of $\Omega $. 
Let $\Omega (n)$ be a sequence of finite sets with $\Omega (n) \subset \Omega (n+1)$ and $\bigcup _n \Omega (n) = \vertexset $.
Since $\Delta $ is bounded on $l^2$,  the symmetry \eqref{symmetry} extends to self-adjointness.
If $f \in l^1$, then  $f \in l^2$, with 
\[ \sum_{v \in \vertexset} [\sg (t) - I] f(v) \wt (v) =  \lim_{n \to \infty} \sum_{v \in \Omega (n) } [\sg (t) - I] f(v) \wt (v) \]
\[ =  \lim_{n \to \infty} \langle \int_0^t \frac{d}{d\tau } \sg (\tau ) f(v) \ d\tau ,1_{\Omega (n)} \rangle  
=  - \lim_{n \to \infty} \langle  \int_0^t  \sg (\tau ) f(v) \ d\tau , \Delta 1_{\Omega (n)} \rangle .\] 
Since $\Delta $ is bounded on $l^{\infty} $,  the function $\Delta  1_{\Omega (n)}$ is uniformly bounded independent of $n$.
Moreover $\Delta  1_{\Omega (n)}(v) = 0$ unless $v$ is adjacent to a vertex $u$ such that $1_{\Omega } (u) \not=  1_{\Omega } (v)$.
Since  $\int_0^t  \sg (\tau ) f(v) \ d\tau \in l^1$, we find
\[ \lim_{n \to \infty} \langle  \int_0^t  \sg (\tau ) f(v) \ d\tau , \Delta 1_{\Omega (n)} \rangle = 0,\]
and 
\[\sum_v \sg(t) f(v) \wt (v) =  \sum_v f(v)\wt (v).\]  
In particular $\sg (t)$ preserves the $l^1(\graph )$ norm of nonnegative functions.
If $g \in l^{1}$ is written as the difference of two nonnegative functions as above, $g = g^+ - g^-$, then
\[\sum_v |\exp(-t\Delta ) g(v) \wt (v)|  = \sum _v | [\exp(-t\Delta ) g^+ - \exp(-t\Delta ) g^-](v) \wt (v)|\]
\[\le \sum _v[ \exp(-t\Delta ) g^+ (v) + \exp(-t\Delta ) g^-(v)] \wt (v) = \| \exp (-t\Delta )|g| \| _1= \| g \| _1, \]
so $\exp (-t\Delta )$ is a contraction on $l^{1} $.

Since $\exp (-t\Delta )$ is  a contraction on $l^1$ and $l^{\infty}$, the Reisz-Thorin Interpolation Theorem \cite[p. 3]{DaviesHKST}
establishes that $\exp (-t\Delta )$ is a contraction on $l^p$ for $1 \le p \le \infty $.

\end{proof}

\subsubsection{The heat kernel}

The diffusion or heat kernel $\sg (t,u,v)$ is obtained when the functional $\sg (t)f(u)$ is represented as an inner product,
\begin{equation} \label{heatker2}
\sg (t)f(u) = \langle \sg (t,u,\cdot), f \rangle  = \sum_w \sg (t,u,w) f(w) \wt (w) .
\end{equation}
Using the $l^1$ normalized functions $\delta _v$ defined above, 
\begin{equation} \label{heatkern}
\sg (t)\delta _v(u) = \sum_w \sg (t,u,w) \delta _v(w) \wt (w) = \sg (t,u,v)  .
\end{equation}
Since $\sg (t)$ is positivity preserving, $\sg (t,u,v) \ge 0$, with strict inequality if $\graph $ is connected by \eqref{strictpos}.
Using \propref{heatprop},
\[1 = \| S(0) \delta _v \| _1 =  \| S(t) \delta _v \| _1\]
\[ = \sum_u \Bigl ( \sum_w \sg (t,u,w) \delta _v(w) \wt (w) \Bigr ) \wt (u)  
=  \sum_u \sg (t,u,v)  \wt (u).\]

Since $\Delta $ is local, explicit estimates for $\sg (t,u,v)$ are readily obtained  \cite{Dodziuk}.
Let $d_c(u,v)$ be the smallest number of edges in a path from $u$ to $v$.
If $d_c(u,v) = k$, then $\Delta ^j\delta _v(u) = 0$ for $j < k$, so 
\[\sg (t,u,v) = \sg (t)\delta _v(u) = \sum_{n=k}^{\infty} \frac{(-t\Delta )^n}{n!} \delta _v(u) , \quad d_c(u,v) = k,\]
and so for $d_c(u,v) = k$,
\[|\sg (t,u,v)| \le \frac{1}{\wt (v) }\sum_{n=k}^{\infty} \frac{t^n\| \Delta  \| _{\infty}^n}{n!} 
 \le  \frac{1}{\wt (v) }\frac{t^k \| \Delta  \|_{\infty}^k}{k!} \sum_{n \ge k} \frac{k!t^{n-k}\| \Delta  \|_{\infty} ^{n-k}}{n!}.\]
For $0 \le k \le n$, the combination of $\frac{k!}{n!} \le \frac{1}{(n-k)!}$ with the Stirling estimate $k! \ge k^ke^{1-k}$
leads to the following rapid spatial decay result.  

\begin{prop}
Assume that \eqref{lpbnd} holds.
Then for $d_c(u,v) = k$,
\[|\sg (t,u,v)| \le \frac{1}{\wt (v)} \frac{t^k \| \Delta \|_{\infty}^k}{k!} \exp(t \| \Delta  \|_{\infty} ) 
 \le \frac{e}{\wt (v)} \Bigl [\frac{et \| \Delta  \|_{\infty}}{ k}\Bigr ]^k \exp(t \| \Delta  \|_{\infty} ) .\]
\end{prop}

\subsubsection{$\Balg $ is an invariant subspace for $\Delta $}

 Recall that $\alg $ is the algebra  of functions $f:\vertexset \to \real $ such that the set of adjacent vertex pairs $v,w$ with $f(v) \not= f(w)$ is a finite set. 
The Banach space (algebra) $\Balg$ is the closure of $\alg $ in $l^{\infty} $. 
The next result shows that $\Balg $ is an invariant subspace for $\Delta $, and thus for $\exp (-t\Delta )$.
Subsequent results will demonstrate that solutions of \eqref{demodel} in $\Balg $ have regular behavior 'at infinity' in $\graph $. 

\begin{thm} \label{invariance} 
Assume that \eqref{lpbnd} holds.
Then $\Balg $ is an invariant subspace for $\Delta $, and for $t \ge 0$ the operator-valued function $\sg (t)$ is a uniformly continuous
semigroup on $\Balg$.
\end{thm}

\begin{proof}

Pick any $\epsilon > 0$.
For any $f \in \Balg $ there is a $g \in \alg $ with $\| f-g \| _{\infty } < \epsilon $.
The function $\Delta g $ is zero at all but a finite set of vertices, so $\Delta g \in \alg $.
Since $\Delta $ is bounded on $l^{\infty}$,
\[ \| \Delta f - \Delta g \| _{\infty  } \le \| \Delta \| _{\infty} \epsilon ,\]
so $\Delta : \Balg \to \Balg $.  Since $\Delta $ is a bounded operator on $\Balg $, it is \cite[p. 2]{Pazy}
the generator of a uniformly continuous semigroup on $\Balg $

\end{proof}

 \section{The $\Balg $ compactification of $\graph$}

This section treats the relationship between the algebras $\alg $ and $\Balg $ and the topology of $\graph $.
Since $\Balg $ is a uniformly closed subalgebra  of $l^{\infty}$ which contains the identity, 
Gelfand's theory of commutative Banach algebras \cite[p. 210--212]{Lax} alerts us to the existence of a 
compactification $\graphbar $ of $\graph $, the maximal ideal space of $\Balg $, on which $\Balg $ acts as a subalgebra of $C(\graphbar )$, 
the continuous real-valued functions on $\graphbar $.   The results of this section will
identify $\graphbar $, while showing that $\Balg$ can be identified with the entire space $C(\graphbar )$.

For a given edge weight function $\rest : \edgeset \to (0,\infty ) $, define the volume of a graph to be the sum of its edge lengths, 
\[vol _{\rest} (\graph ) = \sum_{[u,v] \in \edgeset} \rest (u,v)  .\]
The compactification $\graphbar $ can be realized by equiping $\graph $ with a suitable
new edge weight function $\rho : \edgeset \to (0,\infty ) $.
If this weight function satisfies the condition $ vol _{\rho} (\graph ) < \infty $, then
$\graphbar $ will be the metric space completion of $\graph $ with respect to the new metric induced by $\rho $. 
This compactification is insensitive to the choice of the edge weights as long as the volume is finite.
Some of the following results do not require the finite volume assumption, so $\rho $ is simply assumed
to be a weight function for $\graph $ unless explicitly constrained.

\subsection{Functions}

\begin{lem}
Any $f \in \alg $ is uniformly continuous on the metric space $\vertexset $ with the metric $d_{\rho }$ induced by $\rho $.
\end{lem}

\begin{proof}
For each $f \in \alg $ there are only finitely many vertices $v \in \vertexset $ such that $f(v) \not= f(u)$ for some
$u \sim v$.  Each $v$ has finitely many adjacent vertices, so there is a 
$\delta > 0 $ such that $d_{\rho} (x,y) < \delta $ for any vertex pair $x,y$ implies the value of $f$ does not change
along a path of length less than $\delta $ from $x$ to $y$, so $|f(x) - f(y)| = 0$. 
\end{proof}

\begin{lem} \label{finran}
If $\graph $ is connected, any $f \in \alg $ has finite range.
\end{lem}

\begin{proof}
Suppose $f \in \alg $ is not constant, and suppose $u \in \vertexset $. 
Find a path $(u = v_0,v_1,\dots ,v_N)$ such that
$f(v_n) = f(u)$ for $n \le N$ and $f(v_N) \not= f(w)$ 
for some $w$ adjacent to $v_N$.   
Since the set of such vertices $v_N$ is finite, 
$f(u)$ has one of a finite set of values.
\end{proof}

The metric $\rho $ induces the discrete topology on $\vertexset $.  With this topology, $\vertexset $ is locally compact and Hausdorff.   
As the uniform limits of uniformly continuous functions,  the functions $f \in \Balg $ 
are uniformly continuous and bounded, leading \cite[p. 149]{Royden} to the next lemma.

\begin{lem}
Functions $f \in \Balg$ have a unique continuous extension to $\graphbar _{\rho }$, the metric completion of $\graph $ with
the metric $d_{\rho }$.
\end{lem}

\subsection{Topology}

\begin{prop} \label{compact}
If $vol _{\rho }(\graph ) < \infty $, then $\graphbar _{\rho} $ compact.
\end{prop}

\begin{proof}
If $vol _{\rho}(\graph ) < \infty $, 
then for every $\epsilon > 0$ there is a finite set $V$ of vertices
such that $d_{\rho}(u,V) < \epsilon $ for all $u \in \vertexset $.  That is, the completion of $\graph $ with respect
to the metric $d_{\rho}$ is totally bounded, and $\graphbar _{\rho } $ is thus compact \cite[p. 156]{Royden}.
\end{proof}

The structure of the compactification $\graphbar _{\rho}$ coming from a choice of edge weights with $vol _{\rho }(\graph ) < \infty $
varies wildly with the initial network $\graph $.  If $\graph $ is the integer lattice $\Z ^d$, $\graphbar _{\rho}$ will be its one point compactification.
If $\graph $ is an infinite binary tree, $\graphbar _{\rho}$ will include uncountably many points.

Modifying ideas from \cite{Carlson08}, say that $\graphbar _{\rho} $ is {\it weakly connected} 
if for every pair of distinct points $u,v \in \graphbar _{\rho}$ 
there is a finite set $W$ of edges in $\graph $ such that every path from
$u$ to $v$ contains an edge from $W$.   Trees provide a class of examples with weakly connected completions.
One may extend $\graph $ to a metric graph $\graph _m$ 
by identifying the combinatorial edge $[u,v]$ with an interval of length $\rho(u,v)$.  
By this device some of  the results of \cite{Carlson08} carry over to the present context.  

Finite volume graphs also have weakly connected completions \cite{Carlson08}. 
For a different approach, see \cite{Georg11}.

\begin{thm} \label{weakconn}
If $vol _{\rho }(\graph ) < \infty $ then $\graphbar _{\rho } $ is weakly connected.
\end{thm}

\begin{proof}
The main case considers distinct points $x$ and $y$ in $\graphbar _{\rho} \setminus \graph $.
Remove a finite set of edges from $\edgeset $ so that the remaining edgeset $\edgeset _1$ satisfies 
\[ \sum_{[v_1,v_2] \in \edgeset _1} \rho (v_1,v_2) < \frac{d_{\rho} (x,y)}{2}.\] 
Proceeding with a proof by contradiction,
suppose $(\dots ,v_{-1},v_0,v_1,\dots )$ is a path 
from $x$ to $y$ using only edges in $\edgeset _1$.
Then for $n$ sufficiently large $d_{\rho} (v_{-n},v_n) > \frac{d_{\rho}(x,y)}{2}$,
but there is a simple path from $v_{-n}$ to $v_n$ using only edges from $\edgeset _1$,
so $d_{\rho}(v_{-n},v_n) < \frac{d_{\rho}(x,y)}{2}$.
\end{proof}

Suppose $W$ is a nonempty finite set of edges in $\edgeset $.
For $x \in \graphbar _{\rho }$, let $U_W(x)$ be the set of points $y \in \graphbar _{\rho}$
which can be connected to $x$ by a path containing no edge of $W$. 

\begin{lem} \label{cloplem}
For all $x \in \graphbar _{\rho } $ the set $U_W(x)$ is both open and closed in $\graphbar _{\rho} $.
\end{lem}

\begin{proof}
Take $\epsilon > 0$ such that $2\epsilon < \rho (v_1,v_2)$ for all edges $[v_1,v_2] \in W$.
Suppose $y$ and $z$ are points of $\graphbar _{\rho }$ with $d_{\rho} (y,z) < \epsilon $, 
so there is a path $\gamma _1$ of length smaller than $\epsilon $ from $y$ to $z$.  If 
$x$ is a point of $\graphbar _{\rho }$, and $y \in U_W(x)$, then 
there is a path $\gamma _2$ from $x$ to $y$ containing no edge from $W$.
For $j = 1,2$ there are then vertices $u_j \in \vertexset $ with $u_j \in \gamma _j$ 
such that $d_{\rho }(u_j,y) < \epsilon /2$.   There is a path running from $x$ to $u_1$,
from $u_1$ to $u_2$, and from $u_2$ to $z$ containing no edge from $W$.
This shows that $U_W(x)$ is open.  The argument is the same for each connected 
component of the complement of $U_W(x)$, which is thus also open.
\end{proof}

\begin{thm} \label{totally} 
A weakly connected $\graphbar _{\rho} $ is totally disconnected.
\end{thm}

\begin{proof}
Suppose $v_1$ and $v_2$ are distinct points in $\graphbar _{\rho}$, with 
$W$ being a finite set of edges in $\edgeset $ such that every path from
$v_1$ to $v_2$ contains an edge from $W$.    
The set $U_W(v_1)$ is both open and closed.  
Since $U_W(v_1)$ and $U_W(v_2)$ are disjoint, 
$U_W(v_2) \subset U_W^c(v_1)$, the complement of $U_W(v_1)$ in $\graphbar _{\rho}$. 
Thus $v_1$ and $v_2$ lie in different connected components. 
\end{proof}

If $\graphbar _{\rho}$ is totally disconnected and compact, it
has a rich collection of clopen sets, that is sets which are both open and closed.
In fact \cite{Carlson08} or \cite[p. 97]{Hocking} 
for any $x \in \graphbar _{\rho}$ and any $\epsilon > 0$
there is a clopen set $U$ such that $x \in U \subset B_{\epsilon }(x)$.  In particular
any compact subset of $\graphbar _{\rho}$ can then be approximated by a clopen set.

The next result shows that $\alg $ separates points of $\graphbar _{\rho}$ if and only
if $\graphbar _{\rho}$ is weakly connected.

\begin{thm} \label{separate}
Suppose $\graphbar _{\rho}$ is weakly connected.  If $x$ and $y$ are distinct points of
$\graphbar _{\rho}$, there is a function $f \in \alg $ whose range is $\{ 0 ,1 \}$
such that $f(z) = 0$ for $z$ in an open neighborhood $U$ of $x$, and  
$f(z) = 1$ for $z$ in an open neighborhood $V$ of $y$. 

Conversely, if $\alg $ separates points of $\graphbar _{\rho}$, then 
$\graphbar _{\rho}$ is weakly connected. 
\end{thm}

\begin{proof}
Let $W$ be a finite set of edges in $\graph $ such that every path from $x$ to $y$ 
contains an edge from $W$.  By \lemref{cloplem} the set $U_W(x)$ is both open and
closed in $\graphbar _{\rho}$, as is $U_W^c(x)$.
Define $f(z) = 0$ for $z \in U_W(x)$ and $f(z) = 1$ for $z \in U_W^c(x)$.
For every vertex $v$ and adjacent vertex $w$ we have $f(v) = f(w)$ 
unless $[v,w] \in W$.  Since $W$ is finite, $f \in \alg $.

In the other direction, suppose $\alg $ separates points of $\graphbar _{\rho}$.
Let $x$ and $y$ be distinct points in $\graphbar _{\rho}$, and suppose $f \in \alg $
with $f(x) < f(y)$.  Let $W$ be the finite set of edges $[u,v]$
such that $f(u) \not= f(v)$.  If $\gamma $ is any path starting at $x$ which 
contains no edge from $W$, then $f$ must be constant along $\gamma $.
That is, every path from $x$ to $y$ must contain an edge from $W$, so
$\graphbar _{\rho}$ is weakly connected. 

\end{proof}

\begin{cor} \label{sepsets}
Suppose $\graphbar _{\rho}$ is weakly connected.  If $\Omega $ and $\Omega _1$
are nonempty disjoint compact subsets of $\graphbar _{\rho}$, then there is a function
$f \in \alg $ such that $0 \le f \le 1$,
\[f(x) = 1, \quad x \in \Omega , \quad f(y) = 0, \quad y \in \Omega _1.\]
\end{cor}

\begin{proof}
First fix $y \in \Omega _1$.
Using \thmref{separate}, find a finite cover $U_1,\dots ,U_N$ of $\Omega $ 
by open sets with corresponding functions $f_1,\dots ,f_N$ which satisfy
$f_n(z) = 1$ for all $z$ in some open neighborhood $V_y$ of $y$, while 
$f_n(z) = 0$ for all $z$ in $U_n$.  

Define $F_y = f_1\cdots f_N$.
Find a finite collection $F_1,\dots, F_N$ whose corresponding open sets
$V_1,\dots ,V_N$ cover $\Omega _1$.  The function 
$f = (1-F_1)\cdots (1-F_N)$ has the desired properties.
\end{proof}

\subsection{The maximal ideal space of $\Balg$}

The combination of \thmref{separate} and the Stone-Weierstrass Theorem \cite[p. 212]{Royden} yields the next result.

\begin{thm} \label{Stone}
If $\graphbar _{\rho}$ is weakly connected and compact, then $\alg $ is uniformly
dense in the continuous functions on $\graphbar _{\rho}$.
\end{thm}

 \begin{figure} \label{bignet}
\begin{center}
\includegraphics[width=6in]{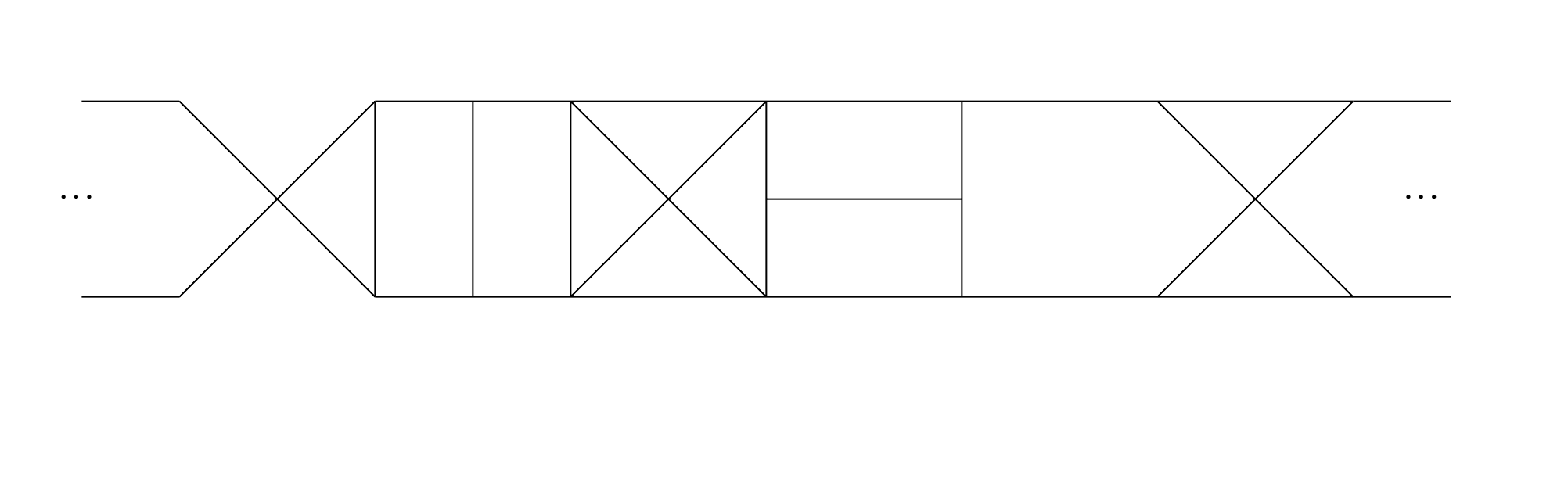}
\caption{Original network} 
\end{center}
\end{figure}

\begin{figure} \label{newnet}
\begin{center}
\includegraphics[width=4in]{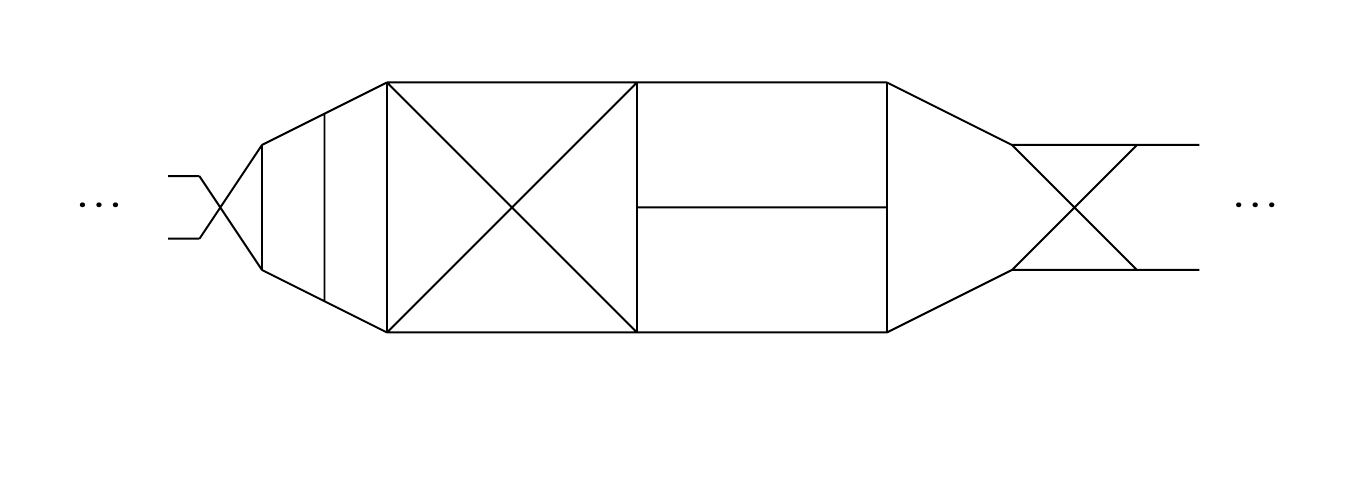}
\caption{Compressed network} 
\end{center}
\end{figure}

Figure 1 provides an illustration of a network equipped with edge weights $R$.
Put new weights $\rho (u,v) > 0$ on the edges of $\graph $, subject to the condition 
\[vol _{\rho} (\graph ) = \sum_{[u,v] \in \edgeset} \rho (u,v) < \infty ,\]
leading to a compressed network as shown in Figure 2.
By \propref{compact} and \thmref{weakconn} the completion $\graphbar _{\rho}$ of $\graph $ with respect to the
new metric is both weakly connected and compact. 
Each function $f \in \Balg $ has a unique continuous extension to a function $f:\graphbar _{\rho } \to \real $.
By \thmref{Stone} this extension map from $\Balg $ to $C(\graphbar _{\rho})$, the real continuous functions on $\graphbar _{\rho}$, 
is surjective, so identifies $\Balg $ with $C(\graphbar _{\rho})$.  It follows \cite[p. 210]{Lax} that 
$\graphbar _{\rho}$ is the maximal ideal space of $\Balg $.   Notice that the particular choice of $\rho $ was not important.
In summary, the next theorem holds. 

\begin{thm}
Suppose $\rho:\edgeset \to (0,\infty )$ is an edge weight function such that $\graphbar _{\rho}$ is compact and weakly connected,
which holds if $vol _{\rho}(\graph ) < \infty$.  Then the continuous extension of $f \in \Balg$ to $f \in C(\graphbar _{\rho }) $
is a surjective isometry of Banach algebras.
\end{thm}

\subsection{More invariant subspaces for $\Delta $}

The Banach algebra $\Balg $ often has a rich collection of  closed ideals which are invariant subspaces for $\Delta $.
Select edge weights $\rho $ satisfying $vol _{\rho}(\graph ) < \infty$.
Define the boundary of $\graphbar _{\rho }$ by
\begin{equation} \label{bndry}
\partial  \graphbar _{\rho} = \graphbar _{\rho} \setminus \vertexset .
\end{equation}
Given a closed set $\Omega \subset \graphbar _\rho $, let 
\begin{equation} \label{ideal}
\Balg _{\Omega } = \{ f \in \Balg , f(x) = 0 {\ \rm for \ all \ } x \in \Omega \}
\end{equation}
denote the closed ideal of functions in $\Balg $ which vanish on $\Omega $. 
The following refinement of  \thmref{invariance} holds when $\Omega \subset \partial \graphbar _{\rho }$. 

\begin{thm} \label{invariance2} 
Assume that \eqref{lpbnd} holds, and 
$\Omega $ is a nonempty closed subset of $\partial \graphbar _{\rho }$.  The ideal $\Balg _{\Omega }$ 
is an invariant subspace for $\Delta $.  If $\Omega _0 \not= \Omega $ is another nonempty closed subset of $\partial \graphbar _{\rho }$,
then $\Balg _{\Omega _0} \not= \Balg _{\Omega }$. 
\end{thm}

\begin{proof}

Suppose $f \in \Balg _{\Omega }$.
For any $\epsilon > 0$ there is a $g \in \alg $ with $\| f-g \| _{\infty } < \epsilon $.
Let 
\[ \Omega _1 = \{ x \in \graphbar _{\rho } , \ |f(x) | \ge \epsilon \} .\]
Since $\Omega _1$ is closed and $\graphbar _{\rho }$ is compact, $\Omega _1$ is compact.
The sets $\Omega $ and $\Omega _1$ are disjoint, so  by \corref{sepsets} there is a $\phi \in \alg $
with $0 \le \phi \le 1$, $\phi (x) = 1$ for $x \in \Omega _1$, and  $\phi (y) = 0$ for $y \in \Omega $.
The function $\phi g \in \alg $, and for all $x \in \graphbar _{\rho }$ either $\phi g = g$,
or $|f(x)| < \epsilon $ and $|\phi (x) g(x)| \le |g(x)| \le 2\epsilon $.  Thus 
$ \| f - \phi g \| _{\infty} <  3 \epsilon $, so $f$ may be uniformly approximated by functions 
$g \in \alg \cap \Balg _{\Omega }$.

If $g \in \alg \cap \Balg _{\Omega }$, then $\Delta g $ is zero at all but a finite set of vertices.
But $\Omega \subset \partial \graphbar _{\rho }$,  so $g(x) = 0$ for all $x \in \Omega $.
This shows that $\Delta : \alg \cap \Balg _{\Omega} \to \Balg _{\Omega }$.
Since $\Delta $ is bounded on $l^{\infty }(\graph )$, while $ \alg \cap \Balg _{\Omega} $ is uniformly dense in $\Balg _{\Omega }$,
it follows that $\Delta : \Balg _{\Omega} \to \Balg _{\Omega }$.

If $\Omega _0 \not= \Omega $ is another closed subset of $\partial \graphbar _{\rho }$, then there is a point $y \in \Omega _0$ such 
that $y \notin \Omega $.  By \corref{sepsets} there is a $\phi \in \Balg _{\Omega } $ with $\phi (y) = 1$, so 
$\Balg _{\Omega _0} \not= \Balg _{\Omega }$. 
\end{proof}

\section{Population Models on Networks}

\subsection{Semilinear equations}

Before treating problems exemplified by \eqref{demodel}, which may arise in models for nonlinear population dynamics,
some basic material for evolution equations in a Banach space $X$ of the form 
\begin{equation} \label{demodel2}
\frac{dp}{dt}  + Ap =  J(t,p), \quad p(0) = p_0.
\end{equation}
will be reviewed.  The operator $-A :X \to X$ is assumed to be the generator of a strongly
continuous semigroup $\sg (t)$ on $X$.   
While the semigroups considered so far have had bounded generators,
cases with unbounded generators $A$ will be encountered shortly.
The remarks below generally follow the treatment in  \cite[pp. 183--205]{Pazy}.

Sufficiently strong assumptions on $J$ lead to a satisfactory global existence theory for \eqref{demodel2}.
The function $J: [0,t_1] \times X \to X$ is assumed continuous.
In addition, suppose $J$ satisfies the uniform Lipschitz condition 
\begin{equation} \label{Lipschitz}
 \| J(t,f) - J(t,g) \| _X \le C_1 \| f - g \| _X , \quad 0 \le t \le t_1,
 \end{equation}
with the constant $C_1$ independent of $t,f,g$. 

For $t_1 > 0$,  a function $p:[0,t_1) \to X$ is a classical solution of \eqref{demodel2} on $[0,t_1)$ if
(i) $p$ is continuously differentiable with $p(t)$ in the domain of $A$ for $0 < t < t_1$,
(ii) $p$ is continuous on $[0,t_1)$, with $p(0) = p_0$, and
(iii) \eqref{demodel2} is satisfied for $0 < t < t_1$.   A classical solution $p$ of \eqref{demodel2} will
satisfy \cite[p. 183]{Pazy} the integral equation
\begin{equation} \label{solnform}
p(t) = \sg (t)p_0 + \int_0^t \sg (t-s)J(s,p(s)) \ ds .
\end{equation}
Solutions of \eqref{solnform} are called mild solutions of \eqref{demodel2}.

With the given hypotheses, the usual iteration method demonstrates that for each $p_0 \in X$ the equation \eqref{demodel2} 
has a unique mild solution $p(t)$ on $[0,t_1]$, and the mapping from $p_0$ to $p(t)$ is Lipschitz continuous from $X$ 
to the continuous functions with values in $X$ with the norm $\| p \|_{\infty} = \sup_{0 \le t \le t_1} \| p(t) \| _X$. 
If in addition $J$ is continuously differentiable, then $p$ is a classical solution of \eqref{demodel2} whenever
$p_0$ is in the domain of $A$.  Additional material, including local existence theorems are in \cite{Pazy}. 

\subsection{Spatial Asymptotics for Population Models}

To handle our population models on networks, additional assumptions are added to the above discussion of abstract semilinear
evolution equations.    The Banach space $X$ will be $\Balg ^d$ for some integer $d > 0$, with
\[\| f \| _X = \max_{1 \le i \le d} \| f_i \| _{\infty}, \quad f = [f_1,\dots ,f_d].\]
The operator $A$ will have the form
\[A = \begin{pmatrix} \Delta _1 & 0 & 0 & \dots & 0 \cr
0 & \Delta _2 & 0 & \dots & 0 \cr
0 &  \vdots & \vdots & \dots & 0 \cr
0 & 0 & 0 & \dots & \Delta _d 
\end{pmatrix} .\]
Each $\Delta _i$ has the form \eqref{Lapop}; the edge and vertex weights may vary with $i$, although they are assumed to 
yield bounded operators $\Delta _i$ defined on the same graph.  
Using vector functions and operators $\Delta _i$ acting diagonally is a minor change, so the previous notation
for the semigroup $\sg (t)$, and the generator $-\Delta = {\rm diag}[-\Delta _1,\dots ,-\Delta _d]$ is maintained. 

Since $\sg _i(t):\Balg \to \Balg $ by \thmref{invariance}, solutions $p(t)$ of  \eqref{solnform} are continuous $\Balg ^d$ 
valued functions.   
The elements of our Banach space are $d$-tuples of continuous functions, which may be evaluated at points $x \in \graphbar _{\rho}$.
Say that the function $J(t,p)$ is determined pointwise if for each $v \in \vertexset $ there is a function $J_v:[0,t_1] \times \real ^d \to \real ^d$ such that 
\begin{equation} \label{Jform}
J(t,p(t))(v) = J_v(t,p(t)(v)), \quad v \in \vertexset .
\end{equation} 
An example is provided by a logistic model varying with both time and vertex,
\[ J_v(t,u) = u(1 - u/K_v(t) ).\]  
Say that a function $J(t,p)$, which is determined pointwise, is eventually constant if the set of edges $[u,v] \in \edgeset $ 
such that $J_u \not= J_v$ is a finite set, independent of $t$.   As in \lemref{finran}, if $\graph $ is connected then there are only finitely many
distinct functions $J_v(t,u)$. 
With these restrictions on $J$, the next result shows that for $x \in \partial  \graphbar _{\rho}$ (recall \eqref{bndry}) 
the value of solutions $p(t,x)$ of \eqref{demodel2} (obtained by continuous extension from $\graph $)
are simply obtained by solving the corresponding ordinary differential equation.

\begin{thm} \label{spatasy}
Assume $J:[0,\infty) \times \Balg ^d \to \Balg ^d$ is continuous for $t \ge 0$ and satisfies the Lipschitz condition \eqref{Lipschitz}.
In addition, suppose $J$ is determined pointwise and eventually constant.
For $p_0 \in \Balg ^d$, assume $p(t)$ is a solution of \eqref{solnform} for $0 \le t \le t_1$.
If $x \in \partial  \graphbar _{\rho} $, and $q(t)$ solves the initial value problem 
\begin{equation} \label{demodel0}
\frac{dq}{dt} = J_x(t,q), \quad q(0,x) = p_0(x),
\end{equation} 
then $p(t,x) = q(t,x)$ for $0 \le t \le t_1$.
\end{thm}

\begin{proof}

Recall, as in the paragraph following \eqref{solnform}, that  the mapping from $p_0$ to $p(t)$ is Lipschitz continuous from $X$ 
to the continuous functions with values in $X$ with the norm $\| p \|_{\infty} = \sup_{0 \le t \le t_1} \| p(t) \| _X$. 
Since $\alg $ is dense in $\Balg $, the result is valid for all $p_0 \in \Balg ^d$ if it holds for all $p_0 \in \alg ^d$.

The equation \eqref{demodel0} is equivalent to the integral equation
\[q(t,v) = q(0,v) + \int_0^t J_v(s,q(s,v)) \ ds ,\]
a system of decoupled evolutions on the vertices of $\graph $, or by
continuous extension on $\graphbar _{\rho} $.
This equation and \eqref{solnform} may be solved by the usual iteration schemes
\begin{equation} \label{iter0}
q_0(t) = p_0, \quad q_{n+1}(t) = p_0 + \int_0^t J(s,q_n(s)) \ ds ,
\end{equation}
and 
\begin{equation} \label{iter1}
p_0(t) = p_0,
\end{equation}
\[ p_{n+1}(t) = \sg (t)p_0 + \int_0^t \sg (t-s)J(s,p_n(s)) \ ds .\]
The hypotheses on $J$ imply that the sequences generated by both schemes \eqref{iter0} and \eqref{iter1}
converge uniformly in $C([0,t_1],\Balg )$ to solutions $q(t), p(t)$ of the corresponding integral equations.  That is, given any $\epsilon _1 > 0$, there is an  
$N$ such that 
\[ \| q(t) - q_n(t) \| _{X} < \epsilon _1, \quad  \| p(t) - p_n(t) \| _{X} < \epsilon _1, \quad n \ge N , \quad 0 \le t \le t_1.\]

Let $\sg _k(t)$ denote the operator valued function defined by the truncated series
\[ \sg _k(t) = \sum_{n=0}^{k} (-t\Delta )^n/n!. \]
Then for any $\epsilon _2 > 0$ there is a $K$ such that 
\[ \| \sg (t) - \sg _k(t) \| _{X} = \| \sg (t) - \sum_{n=0}^{k} (-t\Delta )^n/n!\| _{X} < \epsilon _2, \quad k \ge K, \quad 0 \le t \le t_1.\]
For each $k$, define a new sequence $P_{k,n}:[0,t_1] \to X $ by 
\begin{equation} \label{iter2}
P_{k,0}(t) = p_0,
\end{equation}
\[ P_{k,n+1}(t) = \sg _k(t)p_0 + \int_0^t \sg _k(t-s)J(s,P_{k,n}(s)) \ ds .\]
On the interval $0 \le t \le t_1$ this sequence will also converge uniformly to a continuous $X$-valued function $P_k(t)$.

The difference between $p(t)$ and $P_k(t)$ will satisfy
\[ p(t) - P_k(t)  =  [\sg (t) - \sg _k(t)]p_0 \]
\[+ \int_0^t \sg _k (t-s)[J(s,p(s)) - J(s,P_k(s))] \ ds \]
\[+  \int_0^t [\sg (t-s) - \sg _k(t-s)]J(s,p(s)) \ ds . \]
The Lipschitz condition on $J$ and the uniform convergence of $\sg _k(t)$ to $\sg (t)$ for $0 \le t \le t_1$
imply that for $k$ sufficiently large,
\[ \| p(t) - P_{k}(t) \| _X  \le \epsilon + C \int_0^t \| p(s) - P_{k}(s) \| _X \ ds  , \quad 0 \le t \le t_1,\]
with the constant $C$ independent of $k$.  
By Gronwall's inequality \cite[p. 24]{Hartman}
\begin{equation} \label{pPk}
\| p(t) - P_{k}(t) \| _X \le \epsilon \exp(Ct) , \quad 0 \le t \le t_1,
\end{equation}
so $P_k(t)$ converges uniformly to $p(t)$ for $0 \le t \le t_1$ as $k \to \infty $.

Given an $\epsilon > 0$ pick $k$ sufficiently large that $\| P_k(t) - p(t) \| _X < \epsilon $,
and then pick $n$ sufficiently large that $\| P_{k,n}(t) - P_k(t) \| _X <  \epsilon $ 
and $\| q_n(t) - q(t) \| < \epsilon $, these estimates all valid for $0 \le t \le t_1$.
The difference between the first iterates $P_{k,1}(t)$ and $q_1(t)$ is 
\[P_{k,1}(t) - q_1(t) = [S_k(t) - I]p_0 + \int_0^t [S_k(t-s) - I] J(s,p_0) \ ds ,\]
and in general the difference of iterates is
\[P_{k,m+1}(t) - q_{m+1}(t) = [S_k(t) - I]p_0 + \int_0^t S_k(t-s)J(s,P_{k,m}(s)) - J(s,q_m(s)) \ ds .\]

Because $J$ is determined pointwise and eventually constant, and $p_0 \in \alg $,
there is a finite subgraph $\Xi _0$ of $\graph $ such that $J(s,p_0)$ 
is constant on connected components of $\graph \setminus \Xi _0$.
Because $\Delta $ is local, the definition of $S_k(t)$ implies that 
$[S_k(t-s) - I] J(s,p_0)(v) = 0$ for vertices $v$ whose combinatorial distance from $\Xi _0$ is greater than $k$.
Thus there is a finite subgraph $\Xi _1$ such that $P_{k,1}(t,v) = q_1(t,v)$ for $v \in \graph \setminus \Xi _1$.
The form of $P_{k,m+1}(t) - q_{m+1}(t)$ and induction then show there is a finite subgraph $\Xi _{m+1}$ such that 
$P_{k,m+1}(t,v) = q_{m+1}(t,v)$ for $v \in \graph \setminus \Xi _{m+1}$.

The earlier convergence observations now imply that 
for any $\epsilon > 0$ and any $p_0 \in \alg $, there is a finite set $\Xi _{\epsilon } $ such that 
\[\| p(t,v) - q(t,v) \| < \epsilon , \quad v \notin \Xi _{\epsilon } , \quad 0 \le t \le t_1 \]
and $q(t,v)$ is independent of $v$ on connected components of $\graph \setminus \Xi _{\epsilon }$.
For $x \in \partial \graphbar _{\rho} $,  let $v_m \in \vertexset $ be any sequence converging to $x$ in the $\rho $ metric.
The continuity of $p(t)$ on $\graphbar _{\rho}$ means $p(t,x) = \lim _{v_n \to x} p(t,v_n) = q(t,x)$,
finishing the proof.
\end{proof}

\subsection{Accelerated diffusion models}

\thmref{spatasy} shows that when the problem \eqref{solnform} is solved in $\Balg $, diffusive effects 'disappear at $\infty $' .
This suggests that effective model simplifications may be achieved by modifying $\Delta $ to increase the 
rate of diffusion in remote parts of $\graph $.  Such a modification can be realized by replacing the distant edge weights $\rest (u,v)$ 
with a new set $\rho (u,v)$ satisfying $\sum_{[u,v] \in \edgeset } \rho (u,v) < \infty $.   A related change of vertex weights 
will also be made.  The effect will be to approximate the semigroup generated by $\Delta $ with a 
semigroup generated by an unbounded operator $\Delta _{\rho}$ with compact resolvent.
(The notation hides the choice of vertex weights.)  The operators $\Delta _{\rho}$ can be selected
to respect the invariant subspaces for $\Delta $ identified in \thmref{invariance2}, while 
their eigenfunctions can provide a finite dimensional approximation for the evolution described by \eqref{solnform} on $\graph $.

\subsubsection{A Sobolev space on $\graph $}

The condition \eqref{lpbnd} will now be relaxed so that unbounded operators $\Delta _{\rho}$ with finite volume edge weights may be treated.
Note that \propref{bform} did not require \eqref{lpbnd}.
When $\Delta _{\rho}$ is unbounded it may have many self-adjoint realizations as an operator on $l^2(\graph )$.
A variety of self-adjoint realizations will be constructed using 'Dirichlet' and 'Neumann' conditions on subsets of $\partial \graphbar _{\rho} $.
The next result considers continuous extension of functions to $\graphbar $ 
when the quadratic form of \propref{bform} is finite. 
 
\begin{thm} \label{formcont}
Suppose $\graph $ is connected. 
Using the metric of \eqref{metricdef}, 
functions $f:\vertexset \to \real$ with $B(f,f) < \infty $
are uniformly continuous on $\graph $, and 
so $f$ extends uniquely to a continuous function on $\graphbar _{\rest}$.   
\end{thm}

\begin{proof}
If $v,w \in \vertexset  $ and $\gamma = (v=v_0,v_1,\dots ,v_K = w) $ 
is any finite simple path from $v$ to $w$, then the Cauchy-Schwarz inequality gives
\[|f(w) - f(v)|^2 
 = \Big | \sum_{k} [f(v_{k+1}) - f(v_{k})] 
\frac{C^{1/2}(v_{k+1},v_k)}{C^{1/2}(v_{k+1},v_k)} \Big |^2 \]
\[ \le \sum_{k} [C(v_{k+1},v_k)(f(v_{k+1}) - f(v_{k}))^2] 
\sum_k R(v_{k+1},v_k) \]
\[ \le 2B(f,f) \sum_k R(v_{k+1},v_k). \]
There is a simple path with $\sum_k R(v_{k+1},v_k) \le 2d(v,w)$, so
\begin{equation} \label{formest}
|f(w) - f(v)|^2 \le 4B(f,f)d(v,w), 
\end{equation}
which shows $f$ is uniformly continuous on $\graph $.
By \cite[p. 149]{Royden} $f$ extends continuously to $\graphbar _{\rest} $.   

\end{proof}

For a given set of vertex weights, the bilinear form may be used to define a 'Sobolev style'
Hilbert space $H^1(\wt )$ with inner product 
\[\langle f,g \rangle _{1} = \sum_v f(v)g(v) \wt (v) + B(f,g).\]
Recall that the $\domain _K$ consists of functions $f:\vertexset \to \real $ which are $0$ 
at all but finitely many vertices.  Let $H^1_0$ be the closure of $\domain _K$ in $H^1(\wt)$.

To insure that all functions $f \in \alg $, including the constant function $f =1$,   
are in $H^1(\wt )$ it is necessary to have $\sum_v \wt (v) < \infty $.
If $\rho $ is a finite volume edge weight function, a possible choice is 
to take the vertex weight $\wt _0(v)$ to be half the sum 
of the lengths of the incident edges,
\[\wt _0(v) = \frac{1}{2}\sum_{u \sim v} R(u,v) .\]
This choice makes the vertex measure consistent with the previously defined
graph volume,
\[\wt _0(\graph ) = \sum_{e \in \edgeset} R(e) = {\rm vol (\graph )} .\]
The corresponding Laplacian
\[\Delta _{\rho }f(v) = \wt _0^{-1}(v)\sum_{u \sim v} C(u,v) (f(v) - f(u))
= \frac{2}{\sum_{u \sim v} R(u,v)} \sum_{u \sim v} \frac{f(v) - f(u)}{R(u,v)}\]
resembles the symmetric second difference operator from 
numerical analysis.   The vertex weight $\wt _0$ is typically distinct from 
$\wt (v) = \sum_{u \sim v} C(u,v)$, a choice 
which appears in the study of discrete time Markov chains
\cite[p. 40]{Doyle}, \cite[p. 73]{Liggett}, \cite[p. 18]{Lyons}. 

\begin{lem} \label{h1lem}
If $\graph $ is connected with finite diameter, then there is a constant $C$
such that 
\begin{equation} \label{conteval}
\sup _{v \in \vertexset} |f(v)| \le C\| f \| _{1},
\end{equation}
so a Cauchy sequence in $H^1(\wt )$ is a uniform Cauchy sequence.   The functions
$f$ in the unit ball of $H^1$ are uniformly equicontinuous \cite[p. 29]{Royden}.
\end{lem}

\begin{proof}
Fixing a vertex $v_0$, \eqref{formest} gives
\[|f(v)| \le |f(v_0)| + |f(v) - f(v_0)|  \le
\| f \| _{1} /\sqrt{\wt (v_0)} + 2\| f \| _{1} {\rm diam}(\graph )^{1/2},\]
which is \eqref{conteval}.  The uniform equicontinuity follows  
from \eqref{formest}.
\end{proof}

\begin{thm} \label{zerobnd}
Suppose $\graph $ is connected and has finite diameter.
If $f \in H^1_0(\wt )$, then $f$ has a unique continuous extension to $\graphbar _{\rest} $
which is zero at all points $x \in \partial \graphbar _{\rest} $. 
\end{thm}

\begin{proof}
Any function $f \in H^1_0(\wt )$ is the limit in $H^1(\wt )$ of a 
sequence $f_n$ from $\domain _K$.  The functions $f$ and $f_n$ have unique continuous 
extensions to $\graphbar _{\rest}$ by \thmref{formcont}.  The extended functions $f_n$ 
satisfying $f_n(x) = 0$ for all $x \in \partial \graphbar _{\rest} $. 
By \lemref{h1lem} the sequence $f_n$ converges to $f$ uniformly on $\graph $, 
so the extensions $f_n$ converge uniformly to the extension $f$ on $\graphbar _{\rest}$.  
Thus $f(x) = 0$ for all $x \in \partial \graphbar _{\rest}  $. 
\end{proof}

Let $S_{K}$ denote the operator 
$\Delta _{\rho}$ on $l^2$ with the domain $\domain _K$, the set of real-valued functions   
with finite support.

\begin{prop} \label{Lapdef}
The operator $S_{K}$ is symmetric and nonnegative on $l^2$.
The adjoint operator $S_{K}^*$ on $l^2$ acts by
\[(S_{K}^*h)(v) = \Delta _{\rho }h(v) = \frac{1}{\wt (v)} \sum_{u \sim v} C(u,v) (h(v) - h(u))\]
on the domain consisting of all $h \in l^2$ for which 
$\Delta _{\rho }h \in l^2$.  
\end{prop}

\begin{proof}
The symmetry and nonnegativity of $S_{K}$ are given by \eqref{symmetry}.  
Since $S_{K}$ is densely defined, $S_{K}^*$ is the operator
whose graph is the set of pairs 
$(h,k) \in l^2 \oplus l^2$ 
such that
\[\langle S_{K}f,h \rangle =  \langle f,k \rangle  \]
for all $f \in \domain _K$.  Suppose $f_v = \frac{1}{\wt (v)}\delta _v$
Then for any $h$ in the domain of $S_{K}^*$,
\[k(v) = (S_{K}^*h)(v) = \langle S_{K}f_v,h \rangle  
 = \sum _w [\sum_{u \sim w} C(u,w) (f_v(w) - f_v(u))]h(w)\] 
\[ = \frac{1}{\wt (v)} \sum_{u \sim v} C(u,v) (h(v) - h(u)).\]

\end{proof}

\propref{Lapdef} provides a basic Laplace operator,
the Friedrich's extension \cite[pp. 322-326]{Kato} of $S_{K}$,
whose domain is a subset of $H^1_0(\wt)$ the closure of 
$\domain _K$ in $H^1(\wt)$.
Let $\Delta _D$ denote the Friedrich's extension of $S_{K}$.
Several features of $\Delta _D$ are implied by the condition
$\wt (\graph ) < \infty $.  

\begin{prop} \label{lbnd}
Suppose $\graph $ is connected, with finite diameter and infinitely many vertices. 
If $\wt (\graph ) < \infty $, $f \in {\rm domain}(\Delta _K)$, and
$\| f \|  = 1$,  
then $\Delta _D$ has the strictly positive lower bound
\begin{equation}
\langle \Delta _D f,f \rangle  = B(f,f) \ge \frac{1}{4\wt (\graph ) {\rm diam} (\graph)}, 
\end{equation}
\end{prop}

\begin{proof}
The Friedrich's extension $\Delta _{K,}$ of the nonnegative symmetric operator $S_{K}$
has the same lower bound, so it suffices to consider functions $f \in \domain _K$.
Since $\| f \|  = 1$ there must be some vertex $v$ where $f^2(v) \ge \wt ^{-1}(\graph )$.
Since $f$ has finite support, there is another vertex $u$ with $f(u) = 0$.
An application of \eqref{formest} gives
\[\wt ^{-1}(\graph ) \le f^2(v) = [f(v) - f(u)]^2 \le 4B(f,f)d(u,v).\]
\end{proof}

\begin{prop}
Suppose $\graph $ is connected, $\graphbar _{\rest} $ is compact, and $\wt (\graph )$ is finite.
Let $S_1$ be a symmetric extension of $S_K$ in $l^2$ 
whose associated quadratic form is 
\[ \langle S_1 f,f \rangle  = B(f,f).\]
Then the Friedrich's extension $\Delta _1$ of $S_1$ has compact resolvent.
\end{prop} 

\begin{proof}
The resolvent of $\dop _{1,\wt}$ maps a bounded set in $l^2$ into a bounded set
in $H^1(\wt )$. Suppose $f_n$ is a bounded sequence in $l^2$,
with $g_n = (\dop _1 - \lambda I)^{-1}f_n$.  
By \lemref{h1lem} and the Arzela-Ascoli Theorem \cite[p. 169]{Royden} the sequence $g_n$
has a uniformly convergent subsequence, which converges in $l^2$.
\end{proof}

\subsubsection{Boundary conditions and operators}

In this section 'Dirichlet' (absorbing) and 'Neumann' (reflecting) boundary conditions are used   
to construct nonnegative self adjoint extensions of $S_{K}$.  
The constructed operators extend to semigroup generators
which are positivity preserving contractions on $l^{\infty}$.
Assume that finite volume edge weights $\rho $ and finite measure vertex weights $\mu $
are given.
Given a closed set $\Omega \subset \partial \graphbar _{\rho} $, 
let $\alg _{\Omega }$ denote the
subalgebra of $\alg $ vanishing on $\Omega $.    It was noted in the proof of \thmref{invariance2}
that $\alg _{\Omega }$ is dense in $\Balg _{\Omega }$. 
Let $S_{\Omega }$ denote the operator with domain $\alg _{\Omega }$
acting on $l^2$ by $S_{\Omega }f = \Delta _{\rho }f$. 

By \propref{bform} the operator $S_{\Omega }$ is nonnegative and symmetric, 
with quadratic form $\langle S_{\Omega } f,f \rangle  = B(f,f)$.   
Let $\Delta _{\Omega }$ denote the Friedrich's extension of $S_{\Omega }$,
and note that the domain of $\Delta _{\Omega }$ is a subset of $H^1(\wt )$.
A slight modification of the proof of \thmref{zerobnd} shows that every function $f$
in the domain of $\Delta _{\Omega }$ extends continuously to $\graphbar $
with $f(x) = 0$ for $x \in \Omega $.   As noted in the proof of \thmref{invariance2},
if $\Omega _1 $ and $\Omega _2$ are distinct nonempty closed subsets of $\partial \graphbar _{\rho}$,
there are functions in $\alg _{\Omega _1}$ which are not in $\alg _{\Omega _2}$,
so the operators  $\Delta _{\Omega _1}$ and $\Delta _{\Omega _2}$ have different domains.

Since the operators $\Delta _{\Omega }$ are nonnegative and self-adjoint on $l^2(\wt )$, 
they generate  $l^2(\wt)$ contraction semigroups $\exp (-t\Delta _{\Omega })$. 
Dirichlet form methods  \cite[p. 20]{DaviesHKST} provide additional information
when two conditions are satisfied.  If 
$Quad(\Delta _{\Omega } )$ denotes the domain of $\Delta _{\Omega }^{1/2}$, 
the first condition is that $ f \in Quad (\Delta _{\Omega } ) $ implies 
$|f| \in Quad (\Delta _{\Omega } )$ and $B(|f|,|f|) \le B(f,f)$.
Since the form is
\[B(f,f) = \frac{1}{2}\sum_{v \in \vertexset} \sum_{u \sim v} C(u,v)(f(v) - f(u))^2,\]
the first condition holds for $f \in \alg _{\Omega }$.  If $f \in Quad (\Delta _{\Omega } )$ then
there is a sequence $f_n \in \alg _{\Omega }$ with
\[ \langle f, f \rangle  + B(f,f) = \lim_{n \to \infty} 
 \langle f_n, f_n \rangle  + B(f_n,f_n).\]
It follows that 
\[ \langle |f|, |f| \rangle  + B(|f|,|f|) = \lim_{n \to \infty} 
 \langle |f_n|, |f_n| \rangle  + B(|f_n|,|f_n|),\]
and $B(|f|,|f|) \le B(f,f)$.

The second condition is that if $f \in Quad(\Delta _{\Omega } )$ and $g \in l^2(\wt)$
with $|g(v)| \le |f(v)|$ and $|g(v) - g(u)| \le |f(v) - f(u)|$ for all $u,v \in \vertexset$,
then $g \in Quad(\Delta _{\Omega })$ and $Q(g) \le Q(f)$. This is even more transparent than the first
condition.  Again quoting \cite[p. 12-13]{DaviesHKST}, the following result is established.

\begin{thm}
For $t \ge 0$ the semigroups $\exp (-\Delta _{\Omega }t)$ on $l^2$ 
are positivity preserving $l^p$ contractions for $1 \le p \le \infty $ 
\end{thm}

\subsubsection{Strong convergence}

Assume given a set of edge and vertex weights $\rest , \nu $ satisfying \eqref{lpbnd}, and a set of finite volume edge weights $\rho $,
and vertex weights $\wt $ with $\wt (\graph ) < \infty $.
Pick a vertex $r$ and define edge weights $\rest _n$ with 
\[ \rest _n(u,v) = \Bigl \{ \begin{matrix} \rest (u,v), & \max(d_c(r,u),d_c(r,v)) \le n \cr
 \rho (u,v), & {\rm otherwise} \end{matrix} \Bigr \} \]
 and with vertex weights
  \[ \wt _n(v) = \Bigl \{ \begin{matrix} \nu (v), & d_c(r,v) \le n \cr
 \wt (v), & {\rm otherwise} \end{matrix} \Bigr \}. \] 
 Pick a closed set $\Omega \subset \partial \graphbar _{\rho} $.
Let $\sg (t)$ denote the $\Balg $ semigroup generated by $\Delta $, and let 
$\sg _n(t)$ denote the $l^{\infty} $ semigroup generated by $\Delta _n = \Delta _{\Omega ,n}$, whose
coefficients are determined by $\rest _n$ and $\wt _n$.

\begin{thm} \label{stronglin}
Fix $t_1 > 0$.  For any $f \in \Balg _{\Omega }$,
\begin{equation} \label{converge}
\lim_{n \to \infty } \| \sg (t)f - \sg _n(t)f \| _{\infty} = 0
\end{equation}
uniformly for $0 \le t \le t_1$.
\end{thm}

\begin{proof}
Since $\sg (t)$ and $\sg _n(t)$ are semigroups of contractions on $l^{\infty }$
it suffices to prove the result for $f \in \alg _{\Omega} $.
Given a positive integer $k$, partition $\vertexset $ into a collection of sets of two types. 
The first type is the single set $V_k$ defined to be the set of vertices whose distance to the 
nearest vertex $v$ such that $f(v) \not= f(u)$ for $u \sim v$ is at most $k$.  
Recalling that $f \in \alg $ has finite range, let  $\{ \beta _1,\dots ,\beta _J \}$ be the range of $f$
when the domain is restricted to $\vertexset \setminus V_k$.
The second type consists of the $J$ sets 
\[B_j = f^{-1}(\beta _j) \cap (\vertexset \setminus V_k),\]
so that $f(v) = \beta _j$ for all $v \in B_j$.  

Letting $U_1(t) = \sg (t) f$, define 
\[ U_0(t,v) = \Bigl \{ \begin{matrix}
U_1(t,v), & v \in V_k, \cr
\beta _j, & v \in B_j, 
\end{matrix}
\Bigr \}\]
and
\[ e_1(t,v) = \frac{\partial U_0}{\partial t} + \Delta U_0 = \frac{\partial U_1}{\partial t} + \frac{\partial (U_0 - U_1)}{\partial t} 
+ \Delta U_1 + \Delta (U_0 - U_1)\] 
\[ =  \frac{\partial (U_0 - U_1)}{\partial t}  + \Delta (U_0 - U_1).\] 
Similarly, define 
\[ e_2(t,v) = \frac{\partial U_0}{\partial t} + \Delta _n U_0 .\] 
Given $\epsilon > 0$, the fact that $\Delta $ is bounded on $\Balg $ and $\sg (t)$ is given by a power series
shows, as in the proof of \thmref{invariance}, that for $k$ sufficiently large, $|e_1(t,v)| < \epsilon $ for all $v \in \vertexset$ and $0 \le t \le t_1$.

For the given set $V_k$, choose $n$ large enough that
$\rest _n (u,v) = \rest (u,v)$ for all vertex pairs $u,v$ 
such that $v$ is adjacent to some $w \in V_k$.
Then for each vertex $v \in \vertexset$, either  $R(u,v) = R_n(u,v)$ for all $u \sim v$,
or $U_0(u) = U_0(v)$ for all $u \sim v$.  In either case $\Delta U_0(v) = \Delta _nU_0(v)$, so
 $e_1(t,v) = e_2(t,v)$ for all $v \in \vertexset $.

The function $U_1(t)$ is continuous from $[0,t_1]$ to $\Balg $.  For each $v$, $U_0(t,v)$
is either constant, or agrees with $U_1(t,v)$, so $U_0(t)$ and $e_1(t,v)$ are also continuous.
Since $U_0(t)$ satisfies the initial value problem
\[ \frac{\partial U_0}{\partial t} + \Delta U_0 = e_1, \quad U_0(0) = f,\]
it is given by \cite[p. 106]{Pazy}
\[U_0(t) = \sg (t) f + \int_0^t \sg (t-s) e_1(s) \ ds .\]
Similarly, $U_0(t) \in \alg $ is a continuous function from $[0,t_1]$ to the domain of $\Delta _n$, so
\[U_0(t) = \sg _n(t) f + \int_0^t \sg _n(t-s) e_1(s) \ ds .\]

Both $\sg (t)$ and $\sg _n(t)$ are semigroups of contractions on $\Balg $, so
the earlier estimate $|e_1(t,v)| < \epsilon $ for all $v \in \vertexset$ and $0 \le t \le t_1$
gives the desired result.
 
\end{proof}

\subsubsection{Semilinear strong convergence}

The convergence result of \thmref{stronglin} can be extended to semilinear equations.
This type of extension seems to be known, but we have not found a convenient reference.

\begin{cor}
The conclusion of \thmref{stronglin} remains valid if $S(t)$ and $S_n(t)$ are the solution operators 
taking initial data in $\Balg ^d$ to the solutions of 
\begin{equation} \label{nonlinevo}
\frac{dp}{dt} + \Delta p = J(t,p), \quad p(0) = p_0, \quad \frac{dP}{dt} + \Delta _nP = J(t,P), \quad P(0) = p_0,
\end{equation}
where, as before, $J:[0,\infty) \times \Balg ^d \to \Balg ^d$ is continuous for $t \ge 0$, and Lipschitz continuous uniformly in $t$ on bounded intervals.
\end{cor}

\begin{proof}
As above, the equation \eqref{nonlinevo} may be recast as the integral equation
\[ p(t) = \sg (t)p_0 + \int_0^t \sg (t-s)J(s,p(s)) \ ds ,\]
which may be solved by iteration.
As in \thmref{stronglin}, assume the sequence of strongly continuous contraction
semigroups $\sg _n(t)$ satisfies
\[ \lim_{n \to \infty } \| \sg (t)f - \sg _n(t)f \| _{\infty} = 0, \quad f \in X\]
uniformly for $0 \le t \le t_1$.  Let $P(n,t)$ be the sequence satisfying 
\[P(n,t) = \sg _n(t) p_0 + \int_0^t \sg _n(t-s) J(s,P(n,s)) \ ds.\]

Consider the difference
\[p(t) - P(n,t) = [\sg (t) - \sg _n(t)]p_0 + I, \]
\[ I = \int_0^t [\sg (t-s)J(s,p(s)) - \sg _n(t-s)J(s,P(n,s)) ] \ ds .\]
By assumption the difference $[\sg (t) - \sg _n(t)]p_0 $ has limit $0$ as $n \to \infty $.
Write $I = I_1 + I_2$, where 
\[ I_1 = \int_0^t [\sg (t-s)J(s,p(s)) - \sg _n(t-s)J(s,p(s)) ] \ ds, \]
\[ I_2 = \int_0^t [\sg _n(t-s)J(s,p(s)) - \sg _n(t-s)J(s,P(n,s)) ] \ ds .\]

Since the function $J(s,p(s))$ is continuous on $[0,t_1]$, for any $\epsilon > 0$ 
there is an $h > 0$ and a piecewise constant approximation
\[J(s,p(s)) \simeq J(s_m,p(s_m)), \quad s_m = mh , \quad s_m \le s \le s_m + h,\]
with 
\[ \| J(s,p(s)) - J(s_m,p(s_m)) \| < \epsilon , \quad s_m \le s \le s_m + h.\]
The strong convergence of $\sg _n(t)$ to $\sg (t)$ uniformly on $[0,t_1]$ then gives
$\lim _{n \to \infty} I_1 = 0$.

Similar to the argument in \cite[p. 184]{Pazy}, the Lipschitz  condition 
\[ \| J(t,f) - J(t,g) \| _X \le L \| f-g \| , \quad f,g \in X, \quad 0 \le t \le t_1,\]
and the fact that the semigroups $\sg _n (t)$ are contractions, means that
for any $\epsilon > 0$ and for $n$ sufficiently large,
\[ \| p(t) - P(n,t) \| \le \epsilon + \int_{0}^t L\| p(s) - P(n,s) \| \ ds .\] 
Gronwall's inequality (\cite[p. 24]{Hartman} or \cite[p. 241]{CoddCar}) then gives
\[ \| p(t) - P_n(t) \| \le \epsilon \exp (Lt ).\] 

\end{proof}

\newpage

\newpage

\bibliographystyle{amsalpha}

\end{document}